\documentclass[12pt,reqno]{article}
\usepackage[ansinew]{inputenc}
\usepackage{amsmath,amsthm,amsfonts,amssymb,amscd,epsfig,graphicx}

\usepackage{amssymb,latexsym}
\usepackage{graphicx,color}
\theoremstyle{plain}

\newtheorem{thm}{\sc {\bf Theorem}}[section]

\newtheorem{lem}{\sc {\bf Lemma}}[section]

\newtheorem{coro}{\sc {\bf Corollary}}[section]

\newtheorem{prop}{\sc {\bf Proposition}}[section]

\theoremstyle{remark}

\newtheorem{rem}[thm]{\sc Remark}

\theoremstyle{definition}

\newtheorem{definition}[thm]{\bf Definition}

\renewcommand{\thefootnote}

\newcommand{\h}{\mathbb H}
\newcommand{\R}{\mathbb R}
\newcommand{\hyp}{{\rm I\!H}}

\newcommand{\tgh}{{\,{\rm tgh}\,}}
\newcommand{\cotgh}{{\,{\rm cotgh}\,}}

\font\sc=cmcsc10

\begin{document}
\title{\bf Hypersurfaces with $H_{r+1}=0$ in ${\mathbb H}^n\times{\mathbb R}$ }
\author{Maria Fernanda Elbert, Barbara Nelli,  Walcy Santos}
\date{}

\maketitle


\begin{abstract}
We prove the existence of rotational  hypersurfaces in $\h^n\times\R$ with $H_{r+1}=0$ and we classify them.
Then we prove   some  uniqueness theorems  for $r$-minimal hypersurfaces  with a given  (finite or asymptotic)  boundary.  In particular, we obtain a Schoen-type Theorem for two ended complete hypersurfaces.

\end{abstract}

\section*{Introduction}

In this article we  deal with $r$-minimal hypersurfaces in $\h^n\times\R,$ that is hypersurfaces in $\h^n\times\R$ with $H_{r+1}=0.$

First we address the  problem of finding  all $r$-minimal  hypersurfaces in $\h^n\times\R$   invariant by rotation with respect to a vertical axis.
We prove that there is   a one parameter family of them and that their  behavior  is very similar to that of  catenoids in $\h^n\times\R$ obtained in   Pierre B\'erard and Ricardo Sa Earp \cite{B-SE} (Theorem \ref{classification-thm}).

 Once proved the existence of this family of examples, we   prove some  rigidity results for $r$-minimal hypersurfaces spanning a fixed boundary or asymptotic boundary.   In particular, we obtain  classification results  provided either the boundary or the asymptotic boundary is contained in two parallel slices  (Theorems \ref{uniqueness-theo} and  \ref{uniqueness-cat}).
For the precise definition of asymptotic boundary,  see   the end of Section \ref{preliminaries}.
  Theorem \ref{uniqueness-theo} is inspired  by the results of Jorge Hounie and Maria Luiza Leite \cite{HL3} for $r$-minimal hypersurfaces in Euclidean space. Theorem  \ref{uniqueness-cat} is what 
  we call a Schoen-type result. In his classical paper \cite{S}, R. Schoen characterizes the minimal complete hypersurfaces which are regular at infinity and have two ends. This result was generalized for $r$-minimal hypersurfaces of the Euclidean space by Levi Lopes de Lima and Antonio Sousa \cite{LS} and also by Maria Luiza Leite and Henrique Ara\'{u}jo \cite{AL}. A Schoen-type result for minimal hypersurfaces in $\h^n\times\R$ was obtained by the second author, Ricardo Sa Earp and Eric Toubiana in \cite{NST}. Our Theorem \ref{uniqueness-cat} is a generalization of the latter. In Euclidean space, the proofs in \cite{S}, \cite{HL3}, \cite{LS} and \cite{AL} use the invariance of the minimality (or r-minimality) condition under ambient space scaling. The geometry of $\h^n\times\R$ obliges one to look for suitable strategies or assumptions in order to obtain the corresponding results. The reader will find more details and comments throughout the text. 
%

 We recall that when working with $H_{r+1}=0$ we are lead to use  a version of the maximum principle
different from the one used for  classical minimal hypersurfaces.
In fact, here, ellipticity is not for free and  one has to add some hypothesis on the principal curvatures  vector  (see Section \ref{uniqueness-section}). One of the consequence of this fact is that we must assume embeddedness in Theorem \ref{uniqueness-cat}, that is for free in the
mean curvature case.

 Hypersurfaces  with $H_{r+1}=0$  in $\R^{n+1}$ have been broached in several papers. We refer the reader to \cite{AL}, \cite{HL1}, \cite{HL2}  and \cite{LS} and the references therein.

The paper is organized as follows. In the first section we fix notations. The second  section is devoted to the classification of $r$-minimal hypersurfaces invariant by rotations and to the establishment of their properties. In section three, we establish our  uniqueness results for
$r$-minimal  hypersurfaces  with either (finite) boundary or asymptotic boundary contained in two parallel slices.

\section{Preliminaries}
\label{preliminaries}

  Let $M^n,$  $\bar{M}^{n+1}$ be  oriented Riemannian manifolds of dimension $n$ and $n+1$ respectively and  let  $X:M^n\to \bar{M}^{n+1}$  be an isometric immersion.  Let $A$ be the linear operator associated to the second fundamental form of $X$ and  $k_1,...,k_ n$ be its eigenvalues. The $r$-mean curvature  $H_{r+1}$ of $X$ is given by
$${\binom n{r+1}}H_{r+1}=\sum_{i_1<...<i_{r+1}}k _{i_1}...k_{i_{r+1}},\; 1\leq r+1\leq n.$$

We recall that $H_1$ ($r=0$) is the mean curvature of the immersion and that $H_n$ ($r+1=n$) is the Gauss-Kronecker curvature. The Newton tensors associated to $X$ are inductively defined by
$$\begin{array}{ll}
P_0= & I,\\
P_{r+1}= & {\binom n{r+1}}H_{r+1}I-A\circ P_{r},\;r>0.
\end{array}$$

For further details about the Newton tensors, see \cite{Re}, \cite{Ro}.
We are  interested in the case  where  $\bar{M}^{n+1}=\hyp^n\times \R$, where $\hyp^n$ denotes the hyperbolic $n$-space and $H_{r+1}=0,$ for some $r.$

\

We use the ball model of the hyperbolic space $\hyp^n$ ($n\geq 2$), i.e.
$$
\hyp^n=\{x=(x_1,\ldots,x_n)\in\R^n|x_1^2+\ldots+x_n^2\leq 1\}
$$
endowed with the metric
$$
g_\hyp:=\frac{dx_1^2+\ldots+dx_{n}^2}{\left( \frac{1-|x|^2}{2}\right )^2}
$$

 In $\hyp^n\times \R$, with coordinates $(x_1,\ldots,x_n,t)$, we consider the product metric $$g_\hyp+dt^2.$$

For later use, we  briefly recall  the notion of asymptotic boundary of a hypersurface. We denote the ideal  boundary of $\h^n\times\R$ by $\partial_{\infty}(\h^n\times\R).$

Since we are using the ball model for $\h^n,$
$\partial_{\infty}(\h^n\times\R)$ is naturally  identified with the cylinder $S^{n-1}_1\times\R$ joined with the endpoints of  all the non horizontal
geodesic of  $\h^n\times\R.$ Here, $S^{n-1}$ denotes the unitary (n-1)-dimensional sphere.

The {\em asymptotic boundary} of a hypersurface $M$ in $\h^n\times\R$ is the set of the limit points of $M$
in $\partial_{\infty}(\h^n\times \R)$ with respect to the Euclidean topology of $S^{n-1}\times\R.$
The asymptotic   boundary of the surface $M$ will be denoted by $\partial_{\infty}M,$  while the usual (finite)
boundary of $M$  will be denoted by $\partial M.$

\section{ $r$-minimal rotational hypersurfaces}

 Our aim in this section is to classify the $r$-minimal  hypersurfaces in $\hyp^n\times \R$ invariant by rotation about a vertical axis.
 In $\hyp^n\times \R,$ we consider  the coordinates $(x_1,\ldots,x_n,t)$  and, up to isometry, we can assume the rotation axis to be
 $\{0\}\times \R$.   Notice that the slices $t=const$ are $r$-minimal hypersurfaces invariant by rotation for any $r.$

 We consider a hypersurface obtained by the rotation of a regular  curve in the vertical plane $V:=\{(x_1,\ldots,x_n,t)\in\hyp^n\times \R|x_1=\ldots=x_{n-1}=0\},$ parametrized by  $(\tanh(\frac{f(t)}{2}),t),$  where  $f$ is a positive function.

We define a {\it rotational hypersurface} in $\hyp^n\times \R$  by the parametrization
$$
X:\left\{\begin{array}{rcl}
\R \times S^{n-1}&\to & \hyp^n\times \R \\[5pt]
(t,\zeta)\;\;\;\;&\to& (\tanh(f(t)/2)\zeta,t).
\end{array}\right.
$$
The normal field to the immersion can be chosen to be

\begin{equation}
\label{normal-vector-rot}
N=(1+f_{t} ^2 (t))^{-1/2}\left(\frac{-1}{2\cosh^2(f(t)/2)}\zeta,f_t (t)\right)
\end{equation}

\

\noindent and the principal curvatures associated to $X$ are then given  by (see \cite{B-SE})
 $$k_1=k_2=...=k_{n-1}=\cotgh(f(t))(1+f_{t} ^2 (t))^{-1/2}\;\;\mbox{and}\;\;k_n=-f_{tt}(t)(1+f_{t} ^2 (t))^{-3/2}.$$

\

\noindent We set $q=\frac{n-r-1}{r+1}$ and a straightforward computation yields

\begin{equation}
 (q+1)H_{r+1}=-\cotgh^{r}(f(t))f_{tt}(t)(1+f_{t} ^2 (t))^{-\frac{r+3}{2}}+q\cotgh^{r+1}(f(t))(1+f_{t} ^2 (t))^{-\frac{r+1}{2}}
\label{ed1}
\end{equation}

or, equivalently,

\begin{equation}
 (q+1)f_{t}(t)(1+f_{t} ^2 (t))^{\frac{r}{2}}\frac{\sinh^{q+r}(f(t))}{\cosh^{r}(f(t))}H_{r+1}=\frac{\partial }{\partial t}\left[\sinh^{q}(f(t))\left(1+f_{t} ^2 (t)\right )^{-\frac{1}{2}}\right].
\label{ed2}
\end{equation}

The solutions of  either  \eqref{ed1} or \eqref{ed2}  with  $H_{r+1}=0$ will be   the profile of the  $r$-minimal hypersurfaces invariant by rotation.

\

We state below our classification result. We point out that in the statement we discard  the slices, that are $r$-minimal for each $r$.

\begin{thm}
\label{classification-thm}
The  $r$-minimal complete  hypersurfaces invariant by rotation in $\h^n\times\R$ are the following:

\begin{itemize}
 \item [a)] For $n=r+1:$  right cylinders above spheres of dimension $n-1.$
 \item [b)] For $r+1<n:$  a  one parameter family  $\{{\mathcal M}_a(r)\}_{a>0}$ of  hypersurfaces with the following properties.
 Any ${\mathcal M}_a(r)$ is  embedded and homeomorphic
 to an annulus symmetric with respect to the slice $t=0$. The distance between the rotational axis and the
 \textquotedblleft neck\textquotedblright of ${\mathcal M}_a(r)$ is $a$. The asymptotic boundary of
${\mathcal M}_a(r)$ is composed by two horizontal circles in $\partial_\infty(\hyp)\times \R$ whose vertical distance
 is an increasing function of $a,$ taking values   in  $\left(0 ,\frac{(r+1)\pi}{(n-r-1)}\right).$
 Moreover, if $a\not=b$ then the generating curves of ${\mathcal M}_a(r)$ and ${\mathcal M}_b(r)$ intersect exactly at two symmetric points.
\end{itemize}

\end{thm}

\begin{proof}

For $n=r+1$, it is easy to see that the solutions of equation (\ref{ed1}) for $H_{r+1}=0$ satisfy $f_t(t)={\rm const}$, that is, they are part of cones or right cylinders. Since we search for complete hypersurfaces, a) is proved.

We now prove b). We first notice that, in order to solve \eqref{ed1} with $H_{r+1}=0$, it is enough to solve
 the following Cauchy problem

\begin{equation}\left\{\begin{array}{lcl}
f_{tt}&=& q\cotgh(f(t))(1+f_{t} ^2 (t))  \\[5pt]
f(0)&=& a \\[5pt]
f_t (0)&=&0,
\label{CP}
\end{array}\right.
\end{equation}

for any $a>0.$

 In fact, we only have to realize that the condition $f_t (0)=0$ is not  restrictive. We recall that the Cauchy-Lipschitz Theorem guarantees the existence of a unique maximal solution for given initial data. Since we are considering $f(t)>0$, a  solution of the equation in (\ref{CP}) satisfy $f_{tt}\geq q>0$. Then, the maximal solution attains a minimum at some point of the corresponding interval. We can, w.l.g., suppose it attains a  minimum at $t=0$ and we are done.

 Let $(I_a,f(a,t))$ be the maximal solution of (\ref{CP}). Since $f(a,-t)$ also solves the equation, we conclude that $f(a,t)$ is an even function of $t,$  and we  can write $I_a=(-L(a),L(a))$ for some $L(a)\in\R^+\cup\{\infty\}$.

By imposing $H_{r+1}=0$ in the equation (\ref{ed2}), integrating and using the initial conditions of the Cauchy problem we obtain

\begin{equation}
\label{CP1}
\frac{\sinh^{q}(f(a,t))} {(1+f_{t} ^2 (a,t))^{\frac{1}{2}}}=\sinh^{q}(a) \ \ \ {\rm for \   all\ }t\in I_a.
\end{equation}
In order to obtain the result, we  explore the geometric properties of the solutions $(I_a,f(a,t)),$  that can be deduced from
  \eqref{CP} and \eqref{CP1}.   Our analysis is inspired by the one  in \cite{B-SE} and \cite{ES}.

Since $f_{tt}(t)>0,$  the profile curve is strictly convex.
Moreover, $f(a,.)$ is greater or equal to $a$ and is increasing on $(0,L(a)).$
As it is a  maximal solution of (\ref{CP}) (and (\ref{CP1})),  $f(a,.)$ must go to infinity for $t\longrightarrow \pm L(a).$
Then, we can define  the inverse function $\lambda(a,\rho)$ for $\rho\in [a,\infty)$ onto $[0,L(a)]$ that  satisfies  $\lambda_\rho(a,f(a,t))f_t(a,t)=1$.
Hence we have

\begin{equation}
 \lambda(a,\rho)=\sinh^q (a)\int_{a}^{\rho} {\frac{1}{\sqrt{\sinh^{2q} (u)-\sinh^{2q}(a)}}}du.
\label{ed3}
\end{equation}

 Setting $v=\frac{\sinh(u)}{\sinh(a)},$  we obtain

\begin{equation}
\label{lambda}
 \lambda(a,\rho)=
 \int_{1}^{\frac{\sinh(\rho)}{\sinh(a)}} (v^{2q}-1)^{-1/2}\sinh(a)(1+v^{2}\sinh^2 (a))^{-1/2}dv.
\end{equation}

Now, we notice that

 \begin{equation}\begin{array}{c}
 \sinh(a)(1+v^{2}\sinh^2 (a))^{-1/2}\leq v^{-1}\\
 \\
{\displaystyle\lim_{a\longrightarrow\infty}  \sinh(a)(1+v^{2}\sinh^2 (a))^{-1/2}= v^{-1}}\\
\end{array} \end{equation}

 and that

 \begin{equation}
 \int v^{-1}(v^{2q}-1)^{-1/2}\;dv=\frac{1}{q}\arctan (v^{2q}-1)^{\frac{1}{2}} +{\rm const}.
 \end{equation}

  \

From the relations above, we obtain  that $\lambda(\rho,a)$  converges at $u=a$ and also when $\rho \to \infty.$ Thus
we can write

 \begin{equation}
 L(a)=\int_{1}^{\infty} (v^{2q}-1)^{-1/2}\sinh(a)(1+v^{2}\sinh^2 (a))^{-1/2}dv.
\end{equation}

Moreover  the limit when $a\longrightarrow \infty$  can be taken under the integral and

\begin{equation}
\lim_{a\to\infty}L(a)=\int_{1}^{\infty} v^{-1}(v^{2q}-1)^{-1/2}dv=\frac{\pi}{2q}=\frac{\pi (r+1)}{2(n-r-1)}.
\end{equation}

 Finally, since

\begin{equation}
\label{L-derivative}
\frac{dL}{da}=\cosh(a)\int_{1}^{\infty} (v^{2q}-1)^{-1/2}(1+v^{2}\sinh^2 (a))^{-3/2}dv>0,
\end{equation}

\noindent we conclude that the function $a\to L(a)$ increases from $0$ to $\frac{\pi (r+1)}{2(n-r-1)}$  when $a$ increases from $0$ to $\infty$.
Since $f(a,t)$ is an even function
of $t,$  we can make a reflection  of the graph of the function $\lambda(\rho,a)$ with respect to the horizontal slice $t=0$ and we obtain a {\em catenary like} curve  with finite height.

 The fact that two  generating curves intersect exactly at two symmetric points follow by considering the function
$\lambda(b,\rho)-\lambda(a,\rho)$ for $a\not=b$ and by using the monotonicity of $L(a)$ (see Figure \ref{picture2}).

\begin{figure}[!ht]
 \centerline{\includegraphics[scale=0.4]{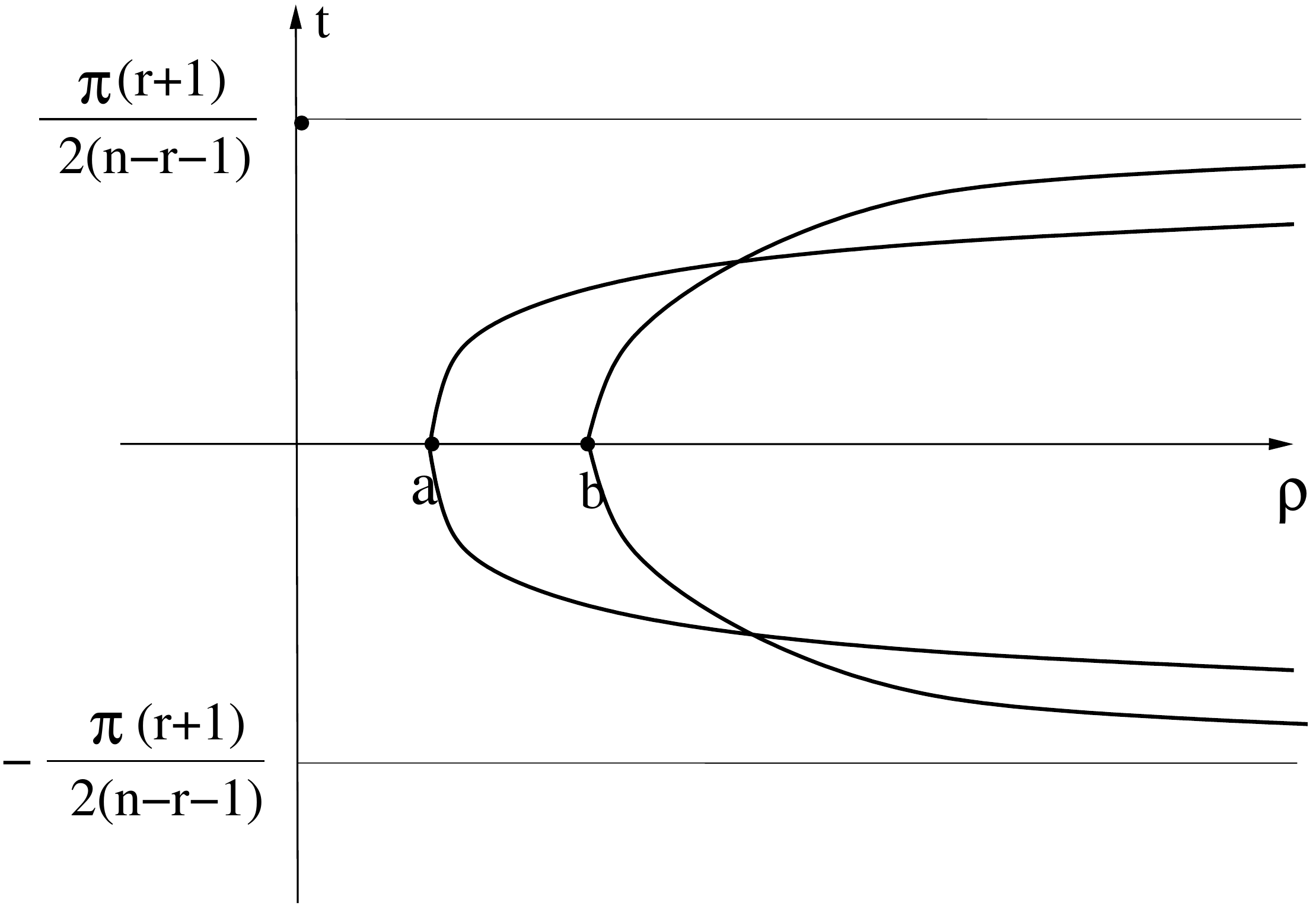}}
 \caption{The profile curves of ${\mathcal M}_a(r)$ and ${\mathcal M}_b(r)$}
 \label{picture2}
 \end{figure}

 With this method we have then found all the complete rotational hypersurfaces that are local graphs over the vertical axis and we are then  able to  conclude that no immersed  examples will appear.

\end{proof}

\begin{definition}
The  elements of the one parameter family  $\{{\mathcal M}_a(r)\}_{a>0}$ of  $r$-minimal complete  hypersurfaces invariant by rotation in $\h^n\times\R$
are called {\em $r$-catenoids}.
\end{definition}

\

In the rest of this section we explore further properties of the family of $r$-catenoids  ${\mathcal M}_a(r)$ (see Figure \ref{picture1}).

 \begin{figure}[!ht]
 \centerline{\includegraphics[scale=0.4]{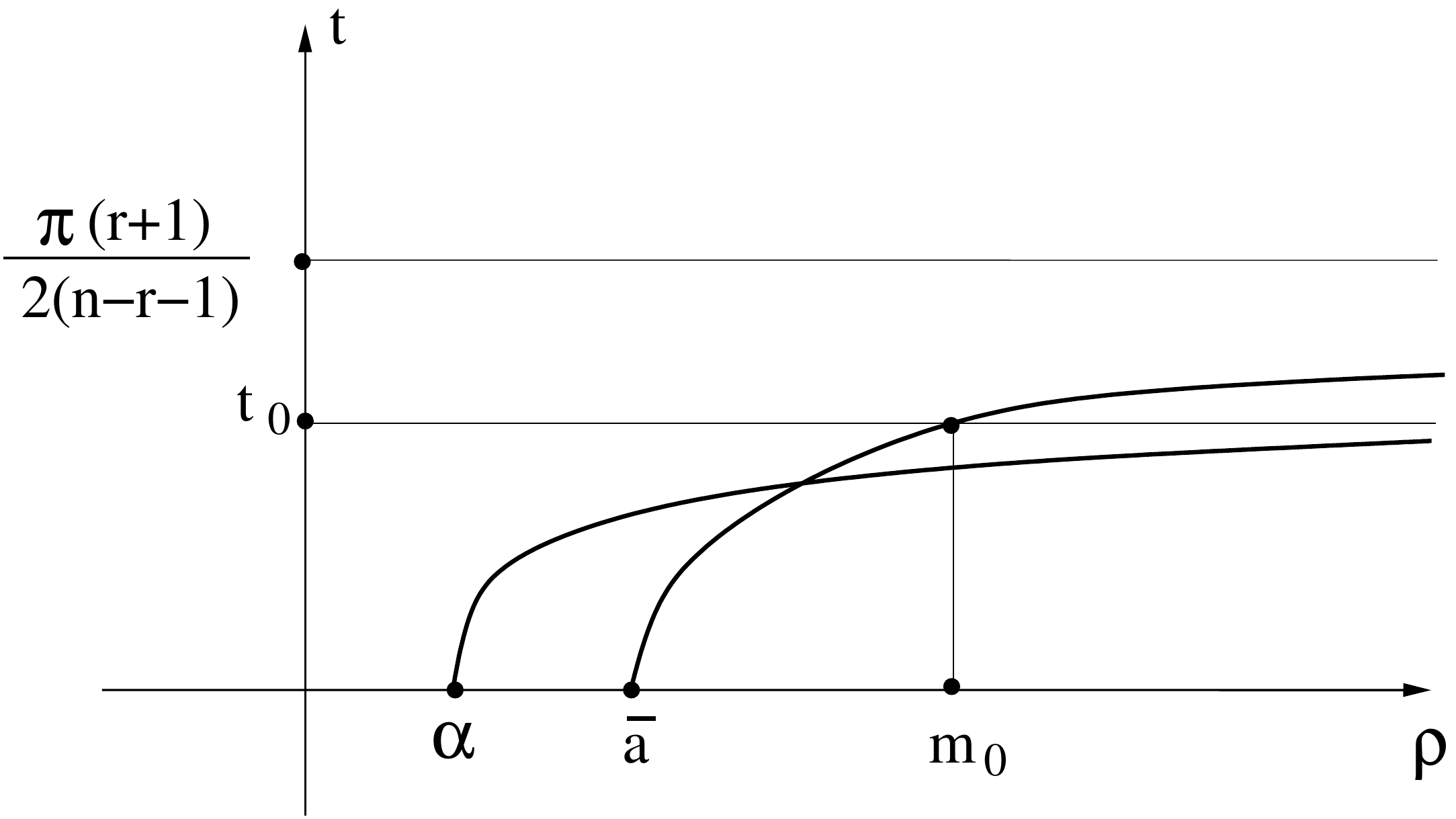}}
 \caption{The minimum for $f.$}
 \label{picture1}
 \end{figure}

Let  us fix $t_0$ in $\left(0, \frac{ \pi (r+1)}{2(n-r-1)}\right)$ and let $\alpha$ be such that $L(\alpha)=t_0.$ This means that
${\displaystyle{\lim_{t\longrightarrow t_0^-}}}f(\alpha,t)=\infty.$ Let $\phi^{t_0}$ be  the positive continuous function defined by $ \phi^{t_0}(a)=f(a,t_0).$
Since, by \eqref{L-derivative}, $\frac{dL}{da}>0,$ we have that $L(a)>L(\alpha)=t_0,$ for any $a>\alpha.$ Then $\phi^{t_0}$ is defined on
$(\alpha, \infty).$ Moreover, $\displaystyle{\lim_{a\longrightarrow \alpha^+}}\phi^{t_0}(a)=\displaystyle{\lim_{a\longrightarrow \infty}}\phi^{t_0}(a)=\infty.$ It is then clear that $\phi^{t_0}$ has a minimum value $m_0$ in $(\alpha, \infty).$ Let $\bar a\in(\alpha,\infty)$ be such that $\phi^{t_0}(\bar a)=m_0.$
Notice that   $f(\bar a,t_0)=m_0$ is a minimum  of $f$ with respect to the variable $a.$

\

\noindent {\bf Claim.}
  $f(\bar a,t_0)=m_0$ is a minimum  of $f$ with respect to the variable $a$ if, and only if,  $\lambda(\bar a,m_0)$ is a maximum of $\lambda,$ with respect to the variable $a.$

\

\noindent {\bf Proof of the claim:}  For, assume that  there exists $\tilde a$ such that $\lambda(\tilde a,m_0)>\lambda(\bar a, m_0).$ Then,
the graph of $\lambda (\tilde a, \rho)$ intersects $t=t_0$ at a point $(\tilde\rho, t_0)$ with $\tilde a<\tilde\rho<m_0.$ Then $f(\tilde a,t_0)=\tilde\rho<m_0=f(\bar a,t_0).$ Contradiction. The proof of the "only if" part is analogous.

\

We now state a technical lemma that will be useful in what follows.

\begin{lem}
\label{lemma-lambda-a}
Let $\lambda(a,\rho)$ be given by \eqref{lambda}. Then we have $\lambda_{aa}(a,\rho)<0$ for $a\in (0,\rho)$ and:
\begin{itemize}
\item [-] $\rho\in (a,\infty)$ if $q\geq 1$.
\item [-]  $\rho\in (0,M]$, where $M={\rm arcosh}\left({\sqrt{\frac{1}{1-q}}}\right)$, if $q<1$.
\end{itemize}
\end{lem}

\begin{proof}

By a straightforward computation, we obtain

\begin{align*}
\lambda_a(a,\rho)&=-\tgh(\rho) \cotgh(a)\left(\left(\frac{\sinh(\rho)}{\sinh(a)}\right)^{2q}-1\right)^{-\frac{1}{2}}
\notag\\
&+\cosh(a)\int_1^{\left(\frac{\sinh(\rho)}{\sinh(a)}\right)}
(v^{2q}-1)^{-\frac{1}{2}}(1+v^{2}\sinh^2(a))^{-\frac{3}{2}}dv\\
\end{align*}

and

\begin{align*}
\lambda_{aa}(a,\rho)&=\frac{\tgh(\rho)}{\sinh^2 (a)}\left(\left(\frac{\sinh(\rho)}{\sinh(a)}\right)^{2q}-1 \right)^{-\frac{3}{2}}.\\[8pt]
&\left[ \left(\frac{\sinh(\rho)}{\sinh(a)}\right)^{2q}\left(1-q\cosh^2(a)-\frac{\cosh^2(a)}{\cosh^2(\rho)}\right) +\left( \frac{\cosh^2(a)}{\cosh^2(\rho)} -1\right) \right]
\notag\\[8pt]
&+\sinh(a)\int_1^{\left(\frac{\sinh(\rho)}{\sinh(a)}\right)}
(v^{2q}-1)^{-\frac{1}{2}}(1+v^{2}\sinh^2(a))^{-\frac{5}{2}}(1-v^2-2\cosh^2(a))dv.\\
\end{align*}

It is easy to see that, under the assumptions, the term $\left(1-q\cosh^2(a)-\frac{\cosh^2(a)}{\cosh^2(\rho)}\right)$ is negative. The remainder terms are clearly negative for $a\in (0,\rho)$.
\end{proof}

\

For any fixed $\rho,$ let $\gamma^{\rho}(a)=\lambda(a,\rho).$

 We can easily see that $\gamma^{\rho}$ is  defined, positive  and continuous in $(0,\rho),$ that
${\displaystyle{\lim_{a\longrightarrow 0^+}}}\gamma^{\rho}(a)=0$ and that $\gamma^{\rho}(\rho)=0.$ Hence, $\gamma^{\rho}$ reaches a maximum at some $a$ in $(0,\rho).$ Set
\begin{itemize}
\item [-] $J_q=(0,\infty)$ if $q\geq 1$
\item [-] $J_q=(0,M]$ if $q<1$.
\end{itemize}
 Lemma \ref{lemma-lambda-a} guarantees that for each $\rho \in J_q$, $\gamma^{\rho}(a)=\lambda(a,\rho)$ has a unique point of maximum.

\noindent When $q<1$, let $A$ be the unique point of maximum of  $\lambda(a,M)$, for $a\in(0,M)$.  We set
\begin{itemize}
\item [-] $T= \frac{ \pi (r+1)}{2(n-r-1)}$ if $q\geq 1$
\item [-]$T=\lambda(A,M)$ if $q<1$.
\end{itemize}

%
%
%
%
%
%
%
%
%

%

\begin{coro}
\label{uniqueness}
For each $t_0\in\left(0, T\right)$, there exists a unique $a_0\in(\alpha,\infty)$ such that $m_0=\phi^{t_0}( a_0)$ is the minimum of $\phi^{t_0}.$ Moreover, for each $\rho> m_0,$ there
exists at least a pair $(a_1,a_2)$, with $a_1<a_0<a_2$, such that $\phi^{t_0}(a_1)=\phi^{t_0}(a_2)=\rho.$

\end{coro}

\begin{proof}

%
%
%
%
%
%

 For $q<1$, it is clear that for each value of $t\in (0,T)$, the minimum value of $\phi^t (a)=f(a,t)$ is less than $M$. Then, taking into account the last claim, we can conclude that  $\phi^t(a)$ has a unique point of minimum since, for each $\rho \in J_q$,  $\gamma_{\rho}$ has a unique point of maximum.  In particular, $\phi^{t_0}$ reaches the minimum value $m_0$ at a unique point, say $a_0.$

Now, we take  $\rho \in J_q$, $\rho>m_0.$ By analyzing the behavior of the profile curves $\lambda (a_0,\rho)$, we see that $\gamma^{\rho}(a_0)=\lambda(a_0,\rho)>\lambda(a_0,m_0)=t_0.$
Since ${\displaystyle{\lim_{a\longrightarrow 0^+}}}\gamma^{\rho}(a)=0$ and that $\gamma^{\rho}(\rho)=0,$ then $\gamma^{\rho}$ reaches the height $t_0$ twice for two values $a_1$ and $a_2$ such that $a_1<a_0<a_2.$ The proof of the Corollary is  now complete.

\end{proof}

\

The following Proposition follows easily from the previous  results. Here, $t_0\in (0,T)$ and $m_0$ and $a_i$, $i=0,1,2$, are the numbers given in Corollary \ref{uniqueness}.

\begin{prop}\label{3-cases} Let $D_+(R),$ $D_-(R)$  two $(n-1)$-spheres of radius $R,$ contained in the slices $t=t_0$ and $t=-t_0$, respectively, with center on
the axis $t.$ We have

\begin{enumerate}
\item[(1)] If $R<m_0,$ there exist no $r$-minimal rotational hypersurfaces with boundary  $D_+(R)\cup D_-(R).$
\item[(2)] If $R=m_0$, there exists a unique $r$-minimal rotational hypersurfaces with boundary \linebreak $D_+(m_0)\cup D_-(m_0), $ namely, ${\mathcal M}_{a_0}(r)$.
\item[(3)] If $R\in J_q$, $R>m_0,$ there exist at least  two  $r$-minimal rotational hypersurfaces with \linebreak boundary $D_+(R)\cup D_-(R).$  Two of them are
${\mathcal M}_{a_1}(r)$ and ${\mathcal M}_{a_2}(r).$

\end{enumerate}

\end{prop}

\

The study of the $r$-catenoids in the Euclidean space was addressed in \cite{HL3}. There, we can see that the vertical heights of the $r$-catenoids are bounded for $q>1$ ($n>2(r+1)$) and unbounded for $q\leq 1$  ($n\leq 2(r+1)$). In $\hyp^n\times\R$, the heights are bounded in both cases. On the other hand, for each admissible value of $t$, the authors in \cite{HL3} were able to prove the uniqueness of the
minimum point of $\phi^t$ by using ambient scaling. Here, by means of geometric arguments and of Lemma \ref{lemma-lambda-a}, we were able to prove the uniqueness for $q\geq 1$, but we fail to prove in the case $q<1$. For $q<1$, we have to restrict the values of $t$  in order to obtain uniqueness. This is, possibly, a technical restriction and we ask the following:

\

{\bf Question:} For any fixed $t_0\in (0, \frac{ \pi (r+1)}{2(n-r-1)})$, we know that there is an $r$-catenoid, ${\cal {C}}$, passing through $(m_0,t_0)$ in the $(\rho, t)$-plane and that all $r$-catenoids passing through $(m,t_0)$ satisfy $m\geq m_0$. Is ${\cal{C}}$ unique?

\

We are able to give a positive answer for $q\geq 1$. For $q<1$, we have to consider $t_0\in (0,T)$. In other words, we ask: can we consider $T=\frac{ \pi (r+1)}{2(n-r-1)}$ for all values of $q$ in Corollary \ref{uniqueness} and in Proposition \ref{3-cases}?

\

\begin{rem}
\label{envelope}
For any $t$ in $(0,\frac{ \pi (r+1)}{2(n-r-1)})$, we define $m(t)$ as the minimum value of the function $\phi^t=f(a,t)$.

Then, the set

\begin{equation*}
{\mathcal E}_{r+1}=\left\{ (m(t)\cos\theta, m(t)\sin\theta, t)\in\h^n\times\R  \ \ | \ \ \theta\in(0,2\pi), \ \ t\in (-T,T)\right\}
\end{equation*}

is  the {\em envelope} of the family ${\mathcal M}_a(r)$, that satisfy $H_{r+1}=0$ { (see Definition 5.16  in \cite{B})}.

\end{rem}

\

Let us state a property or the family ${\mathcal M}_a(r)$ that will be useful in the following section.

\begin{prop}
\label{H-j}
For a fixed $r$, $1\leq r <n-1$, each rotational r-minimal hypersurface of the family ${\mathcal M}_a(r)$ satisfy
\begin{enumerate}
\item [(1)]$H_j>0$, for $j<r+1$.
\item[(2)] $H_{r+1}=0$.
\item [(3)] $H_j<0$, for $r+1<j\leq n$.
\end{enumerate}

\end{prop}

\begin{proof}
By taking \eqref{CP} into account we see that $k_1=\ldots =k_{n-1}$ and that $k_n=-\frac{n-r-1}{r+1}\;k_1$. Then, a straightforward computation yields
$$
H_j=k_1^j\; [(r+1)-j],\;\;\;j=1,\ldots ,n,
$$
that gives the result.
\end{proof}

\

\section{Uniqueness Results}
\label{uniqueness-section}

In this  section we obtain two classification results. The first one  deals with compact $r$-minimal hypersurfaces with boundary on two slices and the
second one  deals with non compact $r$-minimal hypersurfaces with asymptotic boundary   spanned by two copies of $\partial_{\infty}\h^n.$

\

 Before stating the results of this section, we establish some notation.  We denote the slice $\h^n \times \{s\}$, $s\in \R$,  by $\Pi_s$ and  a (closed) slab between two slices by $S$, say  $S=\{(p,t) \mid p\in \h^n,\  t_0\leq t \leq t_1\}$.  The asymptotic boundary of $S$ is given by $\partial_\infty S=\partial_{\infty}\h^n\times[t_0,t_1]$. We set
$\Pi_s^+=\{(p,t) \mid p\in \h^n,\  t>s\}$,
$\Pi_s^-=\{(p,t) \mid p\in \h^n,\  t<s\}$ and, for notational convenience, we write $\Pi=\Pi_0$. Also, we set $\sigma$  for the origin of the slice $\Pi$.

   \

 The complete totally geodesic
hypersurface
${\mathcal P}=\pi\times \R$, where $\pi$ is any totally geodesic \linebreak $(n-1)$-dimensional complete hypersurface
of $\h^n$, is called a {\em vertical hyperplane}.

\

 We will use suitable versions of the interior and boundary maximum  principles for vanishing higher order mean curvatures. We believe it is worthwhile to recall them here and to point out the important differences  between the classical maximum principles for minimal hypersurfaces and these for $r$-minimal hypersurfaces.
For further details about such generalized maximum principles, see \cite{HL1} and \cite{HL2} for hypersurfaces of  Euclidean space  and  \cite{FS} for hypersurfaces of a general Riemannian manifold.

  Let $\overrightarrow{\kappa}=(\kappa_1,...\kappa_n)$ be the principal curvature vector of $M$. Roughly speaking, for $r\geq 1$, the maximum principle requires, as extra hypotheses, that:

\begin{enumerate}
\item[(1)]  the principal curvature vector of the two compared hypersurfaces belong to the same leaf of $H_{r+1}=0$.
\item[(2)]  the rank of the Gauss map (the rank of $\overrightarrow{\kappa}$) of one of the compared hypersurfaces at the contact point is greater than $r.$ This hypothesis guarantees  the ellipticity of the  equation $H_{r+1}=0$ and is satisfied if  $H_{r+2}\neq 0$.

\end{enumerate}

 Let $M,$ $M'$ be two oriented $r$-minimal hypersurfaces of ${\mathbb H}^n\times{\mathbb R}.$
  Let $\overrightarrow{\kappa}$ (respectively $\overrightarrow{\kappa'}$) be the principal curvature vector of $M$ (respectively $M'$).

  {\bf Theorem A.} (Corollary  1.a \cite{FS}) {\em Let $M$ and $M'$ two $r$-minimal oriented hypersurfaces, tangent at a point $p,$ with normal vector pointing in the same direction. Suppose that $M$ remains on one side of $M'$  in a neighborhood of $p.$ Suppose that
  $\overrightarrow{\kappa}(p)$ and  $\overrightarrow{\kappa'}(p)$ belong to the same leaf of $H_{r+1}=0$ and that the rank of either $\overrightarrow{\kappa}$ or $\overrightarrow{\kappa'}$ is at least $r.$ Then $M$ and $M'$ coincide in a neighborhood of $p.$

}

{\bf Theorem B.} (Theorem  2.a \cite{FS}) {\em Let $M$ and $M'$ two $r$-minimal oriented hypersurfaces, tangent at a point $p,$ with normal vector pointing in the same direction. Suppose that $M$ remains on one side of $M'$  in a neighborhood of $p.$ Suppose further that $H_j'(p)\geq 0,$ $1\leq j\leq r$ and either $H_{r+2}\not=0$ or  $H'_{r+2}\not=0.$  Then $M$ and $M'$ coincide in a neighborhood of $p.$ }

  The analogous of both  Theorem A and B  hold for  hypersurfaces  tangent at boundary points (see Corollary 1.b and Theorem 2.b  \cite{FS}).

For the reader's convenience, we
explain here in which cases  either Theorem A or Theorem B   (and their boundary versions) can be used. Then, it will be clear in the following when we use  either the first or the second one.

\begin{itemize}
\item Theorem A will be used for the comparison of an $r$-minimal hypersurface with a reflection of the hypersurface itself.
The assumption of Theorem A are satisfied by a hypersurface and its reflection because of  the following two facts:
\begin{itemize}
\item[{\bf Fact 1:}] Due to properties of hyperbolic polynomials, the principal curvature vector of a  connected hypersurface with $H_{r+1}=0$ and $H_{r+2}\neq 0$ does not change of leaf (see \cite{HL2} for details).

\item[{\bf Fact 2:}] Let $\tau$ be an isometry of $\h^n\times\R$ that preserves the orientation of either $\h^n$ or $\R$ and reverses the other. Let $f:M\to \h^n\times\R$ be an immersion and set $\hat{f}=\tau\circ f$. Then, we have $\hat{N}=-\tau\circ N,$ where $\hat{N}$ is the normal vector to $\hat{f}$   (see \cite{D}, Proposition (3.8)). As a consequence, the second fundamental forms of $f$ and $\hat{f}$  have opposite sign.

\end{itemize}

\item Theorem B will be used for the comparison of an $r$-minimal hypersurface with one of the ${\mathcal M}_a(r)$  that, by Proposition \ref{H-j}, satisfy
$H_j>0$ for $j<r+1$ and $H_{r+2}<0$.
\end{itemize}

  Now, we recall  the description of a family   of hypersurfaces found  by the first author and Ricardo Sa Earp in \cite{ES}, that will be crucial in the proof of Proposition \ref{key-lemma}.  There, the authors proved the existence of a family ${ \cal F}_{\sigma}$ of entire rotational strictly convex graphs with constant $H_{r+1}\in (0,\frac{n-r-1}{n}]$ that satisfy the following properties (see {\cite[Propositions (6.4) and (6.5)]{ES}):
 \begin{itemize}
 \item[(1)] The graphs of the family ${ \cal F}_{\sigma}$ intersect each other only at the point $\sigma$. Moreover, they are  tangent to the slice $\Pi$ at $\sigma$ and have  normal vector pointing upwards.
 \item[(2)] The graphs of the family  ${ \cal F}_{\sigma}$  converge to $\Pi$ uniformly on compact sets as $H_{r+1}$ goes to zero.

 \end{itemize}

  By an isometry of the ambient space, we can produce a new family with an arbitrary common point $q$
  and with normal pointing either upward or downward. We denote by  ${\cal F}_q$ the family with common point $q$ and upward normal vector and by  $\widetilde{\cal{F}}_q$ the one with downward normal vector.

\begin{prop}
\label{key-lemma}
Let $M$ be an  $r$-minimal hypersurface in $\h^{n}\times\R$  such that $\partial M$ and $\partial_\infty M$, one of them possibly empty, are contained in $S\cup \partial_\infty S$, for a given slab $S$. Then $M$ is contained in $S$.




\end{prop}

\begin{proof}
Suppose that $M$ is not contained in the slab $S$. Without loss of generality, we can assume that $S=\{(p,t) \mid p\in \h^n,\  s\leq t \leq 0 \}$ and that there is a subset of $M$ in $\Pi^+$. Now, we choose $\varepsilon>0$ such that $M^+_\varepsilon=M\cap \Pi_{\varepsilon}^+$ is not empty. Since $\partial M$ and $\partial_\infty M$ are in the slab, $M_{\varepsilon}$ is compact with boundary in $\Pi_\varepsilon$. Let $q$ be a point above $M^+_\varepsilon$ and let
$\{\Sigma_i\}_{i\in{\mathbb N}}$  be a sequence of graphs with constant $(r+1)$-mean curvature in the family
$\widetilde{\cal {F}}_q$ that converges to the slice passing through $q$ when $i$ tends to infinity.
Since $M^+_\varepsilon$ is compact, we can suppose, taking a large $i$ if necessary, that $M^+_\varepsilon$ is contained in the  convex side of $\Sigma_i$, for all $i$. Let $l$ be the vertical line passing through $q$. Now, we let $q$ move downwards along $l$ and simultaneously we let $i$ increase. We do this process keeping $M^+_\varepsilon$ in the convex side of the translated $\Sigma_i$, by choosing a subsequence, if necessary. We do this until one of the translated $\Sigma_i$ touches $M^+_\varepsilon$.  Such contact point must be interior and a strictly convex point of $M$. This is a contradiction since $M$ is $r$-minimal.

\end{proof}

\

\begin{coro}
\label{orientation-coro}
Let $M\subset \h^n\times\R$ be a compact embedded $r$-minimal hypersurface with boundary contained in $\Pi_s\cup\Pi_t$, $s<t$ and assume that $\partial M_s=\partial M\cap \Pi_s\neq\emptyset$ and $\partial M_t=\partial M\cap \Pi_t\neq\emptyset$.  Then, $M$ can be oriented by a continuous normal
pointing into the interior of a closed domain $U$ in $\h^n\times\R$, with $M\subset\partial U.$

\end{coro}

\begin{proof}
By the last proposition, we have that $M$ is contained in the slab between $\Pi_s$ and $\Pi_t$. Let $D_s\subset \Pi_s$ and $D_t\subset\Pi_t$ be the bounded region such that $\partial D_s=\partial M_s$ and $\partial D_t=\partial M_t$. Then, $M\cup D_s\cap D_t$ is  an orientable homological boundary  of an (n+1)-dimensional chain in $\h^n\times\R$. We choose the inwards normal to $M\cup D_s\cap D_t$.
\end{proof}

\

\

\
Next Theorem is a uniqueness result for compact $r$-minimal hypersurfaces with boundary in two parallel slices. The analogous result in
the Euclidean space  is Theorem 3.2 in \cite{HL3}.

In the next statement, $t_0$, $m_0$ and $a_0$ are as in Corollary  \ref{uniqueness}. Also, we recall that $D_+(R)$ and  $D_-(R)$ are  two $(n-1)$-spheres of radius $R,$ contained in the slices $t=t_0$ and $t=-t_0$, respectively, with center on
the axis $t.$

\begin{thm}
\label{uniqueness-theo}
Let $M$ be a compact, connected  and embedded r-minimal hypersurface in $\h^n\times\R$, $1\leq r<n-1$, with boundary contained in $\Pi_{t_o}\cup\Pi_{-t_0}$ with $\partial M_+=\partial M\cap \Pi_{t_0}\neq\emptyset$ and $\partial M_-=\partial M\cap \Pi_{-t_0}\neq\emptyset$.
We suppose that $\partial M_+\subset D_+ (m_0)$ and that $\partial M_-\subset D_- (m_0)$. Then, $M$ coincides with the unique rotational hypersurface ${\mathcal M}_{a_0}(r)$ with boundary $D_+(m_0)\cup D_-(m_0).$
\end{thm}
\begin{proof}

By Proposition \ref{key-lemma}, $M$ is contained in the slab $S$ between $\Pi_{t_0}$ and $\Pi_{-t_0}.$ We orient $M$ as  in Corollary \ref{orientation-coro}. As $M$ is compact,  for $a$ large enough,   there exists   a rotational hypersurface ${\mathcal M}_a(r),$ such that $M$ is contained in the compact component determined by ${\mathcal M}_a(r)\cap S.$ Now, we let $a$ decrease. It is clear that there exists $\alpha>0$ such that $ {\mathcal M}_{\alpha}(r)$ has a first contact point with $M.$ We notice that $\alpha\not=0$ because the waist of
$ {\mathcal M}_a(r)$ shrinks to zero as $a\longrightarrow 0$ and the absence of a contact point  before $a=0$ would contradict the connectedness of $M.$

If the first contact point  $p$ between $M$ and ${\mathcal M}_{\alpha}(r)$ is an interior point of $M,$ then $M$ and $ {\mathcal M}_{\alpha}(r)$ are tangent at $p,$ both have normal vectors pointing into the compact region determined by  $M\cap S$  and $M$ lies above
${\mathcal M}_{\alpha}(r)$ with respect to the normal vector (recall that ${\mathcal M}_{\alpha}(r)$ is oriented as in \eqref{normal-vector-rot}).
By Proposition \ref{H-j}, $M_{\alpha}(r)$ is such that $H_{r+2}<0$ and $H_j>0,$ for $j<r+1,$ hence by the maximum principle (see Theorem B), $M$ and
${\mathcal M}_{\alpha}(r)$ coincide in a neighborhood of $p.$ Then, they coincide  everywhere. Moreover, since
$\partial M\subset D_+(m_0)\cup D_-(m_0),$ Proposition \ref{3-cases}  gives the result.

Now, let analyze the case where the first contact point between $M$ and ${\mathcal M}_{\alpha}(r)$ is on $\partial M.$
Let $q\in \partial M$ be a first contact point between $M$ and ${\mathcal M}_{\alpha}(r).$
As $\partial M\subset D_+(m_0)\cup D_-(m_0),$ again by Proposition  \ref{3-cases}, $\alpha=a_0$ and $q$ belongs to
$\partial D_+(m_0)\cup \partial D_-(m_0).$

If the tangent planes at $q$ to $ M$ and $ {\mathcal M}_{a_0}(r)$ coincide, then, by the boundary maximum principle (see Theorem 2.b \cite{FS}, that is the boundary version of Theorem B), $M$ and
$ {\mathcal M}_{\alpha}(r)$ coincide as well and the result is proved.

Otherwise, the slope of $T_qM$ is strictly smaller than the slope of $T_q {\mathcal M}_{a_0}(r).$ We will get a contradiction in this case.
By Proposition \ref{3-cases}, for $\varepsilon$ small, ${\mathcal M}_{a_0-\varepsilon}(r)$ is such that
${\mathcal M}_{a_0-\varepsilon}(r)\cap (\Pi_{t_0}\cup\Pi_{-t_0})$ contains $D_+(m_0)\cup D_-(m_0) $ in its interior.
This  last fact, joint with the fact that the slope of $T_qM$ is strictly smaller than the slope of $T_q {\mathcal M}_{a_0}(r)$ yield that
${\mathcal M}_{a_0-\varepsilon}(r)\cap S$ bounds a region containing $M.$ Now, if we continue decreasing $a,$
$\partial {\mathcal M}_{a}(r)$ can not touch $\partial M$ again (because of Proposition \ref{3-cases}), but for  $a\longrightarrow 0,$ the waist of
${\mathcal M}_{a}(r)$ shrink to zero, so there must be an interior contact point between $M$
and $ {\mathcal M}_{\bar a}(r),$ for some $\bar a<a_0.$ Then, as before, $M$
and $ {\mathcal M}_{\bar a}(r)$ must coincide. This is a contradiction because they have disjoint boundaries.

\end{proof}

\
 Theorem \ref{uniqueness-theo} implies the following result (with the same notation as there).

\begin{coro}
There is no compact, connected and embedded r-minimal   hypersurface in $\h^n\times\R$, $1\leq r<n-1$, with $\partial M_+\subset D_+ (R)$ and $\partial M_-\subset D_- (R)$, for $R<m_o$.
\label{coro-normal}
\end{coro}

\

 Theorem \ref{uniqueness-cat} below is inspired by the classical result of Schoen \cite[Theorem 3]{S}. As in the proof of Theorem \ref{uniqueness-theo}, we have to deal with the restrictions imposed by the geometry of $\h^n\times\R$. Here, based on the ideas contained in \cite{NST}, we change the assumption of regular ends at infinity in the Euclidean space by that of asymptotic boundary in two parallel slices in $\h^n\times\R$. The proofs of  Lemma \ref{symmetry-lemma} and  Theorem \ref{uniqueness-cat} are very similar to that of \cite[Lemma 2.1, Theorem 4.2]{NST}. The differences are essentially due  to the differences in
the hypotheses of the maximum principle for minimal and for $r$-minimal hypersurfaces.
Also, we point out that arguments used to prove embeddedness of the minimal immersion in \cite{S} and \cite{NST}  can not be carried out here.
The obstruction is the requirement in the maximum principle, for the $r$-minimal case, that the principle curvature vectors belong to the same leaf.
Then, here, embeddedness is a hypothesis.


\begin{lem}\label{symmetry-lemma}
 Let $\Gamma^+$ and $\;\Gamma^-$ be two $(n-1)$-manifolds in $\partial_{\infty}\h^n \times \R$
which are vertical graphs over $\partial_{\infty} \h^n \times \{0\}$ and such that
$\Gamma^+\subset \partial_{\infty}\Pi^+$ and $\Gamma^-\subset \partial_{\infty} \Pi^-$. Assume that
 $\Gamma^-$ is the  symmetric  of $\Gamma^+$ with respect to $\Pi$. Let $M\subset \h^n \times \R$ be an  embedded, connected,  complete $r$-minimal
hypersurface, $1\leq r<n-1$,
with two ends $E^+$ and $E^-$. Assume that each end is a vertical graph
and that $\partial_{\infty}M= \Gamma^+ \cup \Gamma^-$, that is
$\partial_{\infty} E^+= \Gamma^+$ and $\partial_{\infty} E^-= \Gamma^-$.
Moreover, assume that $H_{r+2}\not=0.$ Then $M$ is symmetric with respect to $\Pi$. Furthermore, each part $M\cap \Pi^\pm$ is a vertical graph.

\end{lem}

\begin{proof}

We denote by $t^+$ the highest $t$-coordinate of $\Gamma^+$.
Since $\partial_{\infty} M=\Gamma^+ \cup \Gamma^-,$ then
Proposition \ref{symmetry-lemma} imply that $M$ is contained in the slab between $\Pi_{t^+}$ and $\Pi_{-t^+}$.

 We now notice that since each end of $M$ is a vertical graph, we can obtain a compact domain $\Omega\in \h^n\times\{0\} $ such that $E^+$ and $E^-$ are graphs over $(\h^n\times\{0\})\backslash\Omega$. We consider the cylinder $C$ over $\Omega$ and we see that $M_C = M\cap C$ is compact and embedded, so it bounds a compact domain $B$. Then an argument similar to that used in the Corollary \ref{orientation-coro} gives that we can orient $M_C$ towards $B$. Since, $M\backslash M_C$ is a graph, we can extend the normal vector continuosly to $M$. In this case, we will say that the whole $M$ is oriented  towards the interior.

 For any $t>0$ we set $M_t^+=M \cap \Pi_t^+$. We denote by
$M_t^{+*}$ the symmetry of  $M_t^+$ with respect to the slice $\Pi_t$.
As $E^+$ is a vertical graph, there exists $\varepsilon >0$ such that  $M_{t^+ -\varepsilon}^+$
is a vertical graph, then  we can start Alexandrov reflection.  We keep doing the Alexandrov reflection with $\Pi_t$, doing $t\searrow 0$.  Here, we recall that reflection with respect to a slice preserves the orientation of $\h^n$ and reverses that of $\R$. Then,  taking Fact 2 into account, the principal curvature vector of   $M_t^{+*}$ with respect to  the suitable  orientation  $-\hat{N}$, is equal to the principal curvature vector of $M_t^{+}$ . By   Fact 1, the principal curvature vectors  of $M_t^{+*}$ and $M_t^-$ belong to the same leaf, hence  we can apply the maximum principle for comparing them.  Theorem A or its corresponding boundary version (Corollaries 1a and 1b of \cite{FS}), gives, for
$t>0$,  that the surface $M_t^{+*}$ stays above $M_t^-$ and that both, $M_t^-$ and $M_t^-$, are vertical graphs. By doing $t\searrow 0$, we obtain that $M_0^+$ is a vertical graph and
that
$M_0^{+*}$ stays above $M_0^-$.

Doing  Alexandrov reflection with slices coming from below,
one has that
$M_0^-$ is a vertical graph and that
$M_0^{-*}$ stays below $M_0^+$, henceforth we get
$ M_0^{+*}=M_0^{-}$.
Thus $M$ is symmetric with respect to $\Pi$ and each component of $M \setminus\Pi$ is
a graph.
This completes the proof.
\end{proof}

\begin{thm}
\label{uniqueness-cat}
Let $M$ be a  complete connected $r$-minimal hypersurface embedded in $\h^{n}\times\R$, $1\leq r<n-1$, with $H_{r+2}\not=0.$ Assume that $M$ has two ends and that each end is a vertical graph whose asymptotic boundary is a copy of $\partial_{\infty} \h^n.$
Then $M$ is isometric, by an ambient isometry to one of the ${\cal M}_a (r).$
\end{thm}

\begin{proof}

Up to a vertical translation, we can assume that the asymptotic boundary of $M$  is
symmetric with respect to
$\Pi:= \h^n \times \{0\}$, say  $\partial_{\infty} M=\partial_{\infty} \h^n \times \{t_0,-t_0\}$ for some $t_0>0$. Then  $\Gamma^{+} :=\partial_{\infty} M \times \{t_0\}$ and $\Gamma^{-} :=\partial_{\infty} M \times \{-t_0\}$. By Proposition \ref{key-lemma}, $M$ is contained in the slab between $\Gamma^+$ and $\Gamma^-.$ By Lemma \ref{symmetry-lemma}, M is symmetric about $\Pi$,
and each connected component of
$M\setminus \Pi$ is a vertical graph. Moreover, at any point of $M\cap \Pi$, the tangent hyperplane to $M$ is orthogonal to $\Pi.$

 Since $M$ is
embedded, $M$ separates $\h^n \times [-t_0,t_0]$ into two connected components.
We denote by $U_1$ the component whose  asymptotic boundary is
$\partial_{\infty} \h^n \times  [-t_0,t_0]$ and by $U_2$ the component such that
$\partial_{\infty} U_2= \partial_{\infty} \h^n \times \{t_0,-t_0\}$.

Let $q_\infty \in \partial_{\infty} \h^n $ and let $\gamma \subset \h^n$ be an oriented
geodesic issuing from $q_\infty$, that is $q_\infty \in \partial_{\infty} \gamma$. Let
$q_0\in \gamma$ be any fixed point. For any $s\in \R$, we denote  by $P_s$ the vertical hyperplane orthogonal
to $\gamma$ passing through the point of $\gamma$ whose  oriented distance
 from  $q_0$ is $s$. We suppose that $s<0$ for any point in the
geodesic segment
$(q_0, q_\infty)$. For any $s\in \R$, we call $M_s (e)$ the part of $M \setminus P_s$
such that $(q_\infty, t_0) , (q_\infty, -t_0) \in \partial_{\infty} M_s(e)$ and let
 $M_s^* (e)$ be the reflection  of $M_s (e)$ about $P_s$.
We denote by
$M_s (d)$ the other part of $M\setminus P_s$ and by $M_s^* (d)$ its reflection
about $P_s$. We recall that this reflection preserves the orientation of $\R$ and reverses that of $\h^n$.  This, enable us to use Theorem A, as we did in Lemma \ref{symmetry-lemma}.

\smallskip

By assumption there exists $s_1<0$ such that for any $s<s_1$ the part
$M_s(e)$ has two connected components and both of them are vertical graphs.
We deduce that $\partial M_s(e)$ has two (symmetric) connected components, each
one being a vertical graph.

\

\noindent {\bf Claim 1.} For any $s<s_1,$ we have that $M_s^* (l)\cap \Pi^+$
stays above $M_s(d)$ and $M_s^* (e)\cap \Pi^-$ stays
below $M_s(d)$. Consequently  $M_s^* (e)\subset U_2$ for any $s<s_1$.

\

\noindent {\bf Claim 2.} We have $ M_\beta^* (e)= M_\beta(d)$. Thus,
given a geodesic $\gamma\subset \h^n$, there exists a vertical
hyperplane $P_\beta$ orthogonal to $\gamma$ such that $M$ is
symmetric with respect to $P_\beta$

 The reader can find analogous claims  joint with their proofs  in \cite[Theorem 4.2]{NST}.
The proofs go exactly in the same way. There, the authors use the classical maximum principle and here we should use Theorem A or its corresponding boundary version.

\

By Claim 2, one has that $M\cap \Pi$ satisfies the assumptions
of \cite[Proposition  4.1]{NST}.
Then $M\cap \Pi$ is a $(n-1)$-geodesic sphere of $\Pi$. Let $a$ be such that
 ${\cal M }_a (r)$ is the rotational $r$-minimal hypersurface through $M\cap \Pi$ and
orthogonal to $\Pi$.  We set
$ {\cal M }_a (r)^+ :={\cal M }_a (r) \cap \{ t>0\} $. Both $ {\cal M }_a (r)^+$ and $M^+$  are vertical along their common
finite boundary $\Sigma$, hence they are tangent along $\Sigma$. We want to show that they coincide. Let $t({\cal M }_a (r))$ (resp. $t(M)$) be the height of the asymptotic boundary
of $ {\cal M }_a (r)^+$ (resp. $M^+$). Suppose, for example, that $t( {\mathcal M}_a(r)) \leq t(M)$. We translate $M^+$ upward so that it stays above $ {\mathcal M}_a(r)^+$. Then we  translate it down till we find the first point of contact. By using Theorem B, or its corresponding boundary version, we conclude that $M^+={\cal M }_a (r)^+$.

The case  $t(M) \leq t({\cal M }_a (r))$ is analogous.  We then conclude that $M= {\cal M }_a (r)$ and the proof is completed.

\end{proof}

\

{\bf Acknowledgements.} The authors warmly thank Ricardo Sa Earp for very useful discussions.
The second author would like to thank the IM-UFRJ for the hospitality during the preparation of this work. The authors were partially supported by
PRIN-2010NNBZ78-009 and CNPQ-Brasil.

\textsc{Maria Fernanda Elbert}

{\em IM - Universidade Federal do Rio de Janeiro

fernanda@im.ufrj.br}

\textsc{Barbara Nelli}

{\em DISIM, Universit\'a dell'Aquila

nelli@univaq.it}

\textsc{Walcy Santos}

{\em IM - Universidade Federal do Rio de Janeiro

 walcy@im.ufrj.br}

\end{document}